\documentclass[12pt]{article}

\usepackage{amsfonts}
\usepackage{graphics}
\usepackage{amssymb}

\usepackage[dvips]{graphicx}

\newtheorem{theorem}{Theorem}[section]
\newtheorem{lemma}{Lemma}[section]

\newtheorem{proposition}[theorem]{Proposition}
\newtheorem{corollary}[theorem]{Corollary}

\newcommand{\R }{ \ensuremath{ \mathbf{R} }}

\newcommand{\K }{ \ensuremath{ \mathcal{K}_{\nabla} }}

\newcommand{\proof}{{\noindent \bf Proof:} }

\newcommand{\eop }{ \hfill $\Box$ }

\begin{document}

\begin{center}

\vspace{1cm}

 {\Large {\bf A note on stochastic calculus in vector bundles \\[2mm]
}}

\end{center}

\vspace{0.3cm}

\begin{center}
{\large { Pedro J. Catuogno}\footnote{E-mail: pedrojc@ime.unicamp.br. Research partially
supported by CNPq, grant no. 302704/ 2008-6, 480271/2009-7 and FAPESP, grant no. 07/06896-5} \ \ \ \ \ \ \ \ \ \ \ \ \  Diego
S. Ledesma\footnote{E-mail: dledesma@ime.unicamp.br. Research supported by FAPESP, grant no. 10/20347-7.}

\bigskip

{ Paulo R. Ruffino}\footnote{E-mail: ruffino@ime.unicamp.br. Research partially
supported by CNPq, grant no. 306264/ 2009-9, 480271/2009-7 and FAPESP, grant no. 07/06896-5.}}

\vspace{0.2cm}

\textit{Departamento de Matem\'{a}tica, Universidade Estadual de Campinas, \\
13.081-970 - Campinas, SP, Brazil.}

\end{center}

\begin{abstract}
The aim of these notes is to relate covariant stochastic  integration in a vector bundle $E$ (as in Norris \cite{Norris}) with the usual Stratonovich calculus via the connector $\K:TE \rightarrow E$ (cf. e.g. Paterson \cite{Paterson} or Poor
\cite{Poor}) which carries the connection dependence. 

\end{abstract}

\vspace{1cm}



\smallskip



\section{Introduction}

Stochastic calculus on vector bundles has been studied by several authors, among others, Arnaudon and Thalmaier \cite{Arnaudon-Thalmaier}, Norris \cite{Norris}, Driver and Thalmaier \cite{Driver}. In these articles, the stochastic integral of a semimartingale $v_t$ in a vector bundle $\pi:E\rightarrow M$ is defined by decomposing $v_t$ into horizontal and vertical (covariant) components according to a given connection in $E$.
The aim of these notes is to relate covariant stochastic  integral in vector bundles (Norris \cite{Norris}) with the usual Stratonovich calculus using an appropriate operator, the connector $\K$ (cf. e.g. Paterson \cite{Paterson} and Poor \cite{Poor}), from the tangent space $TE$ to $E$ which carries the connection dependence.

\bigskip

We denote by $M$ a smooth differentiable manifold. Let $E$ be an $n$-dimensional vector bundle over $M$ endowed with a connection $\nabla$. This connection induces a natural projection
$\mathcal{K}_{\nabla}: TE \rightarrow
E$ called the associated connector (cf. Paterson \cite{Paterson} and Poor
\cite{Poor})  which projects into the vertical subspace of $TE$ identified with $E$. More precisely: Given a differentible curve $v_t \in E$, decompose $v_t=u_t f_t$, where $u_t$ is the unique horizontal lift of  $\pi(v_t)$ in the principal bundle $Gl(E)$ of frames in $E$ starting at a certain $u_0$ with $\pi (u_0)= \pi(v_0)$ and $f_t \in \R^n$. 
Then 
\[
\mathcal{K}_{\nabla} (v'_0) := u_0 f'_0.
\]
Norris \cite{Norris} defines 
the covariant Stratonovich integration of a section $\theta$ in the dual vector bundle $E^*$
 along a process $v_t \in
E$ by:
\[
 \int \theta Dv_t := 
 \int \theta
 u_t \ \circ d f_ t, 
\]
where $v_t = u_t f_t$; and the corresponding covariant It\^o version:
\[
 \int \theta D^I v_t :=  \int \theta
 u_t \ d f_ t.
\]

\section{Main results}

Initially, observe that using the connector
$\K$, the covariant integral above reduces to a classical Stratonovich integral of 1-forms:

\begin{proposition}  Let $v_t$ be a semimartingale in $E$ and $\theta \in \Gamma (E^*)$. Then 
 \[
   \int \theta ~  Dv_t =  \int \theta ~   \K \ \circ d v_t.
 \]
\end{proposition}

\proof Let $\phi: Gl(E) \times \mathbf{R}^n \rightarrow E$ be the action map $\phi(u,f)= uf$. The right hand
 side in the equation above is
\begin{eqnarray}
 \int \theta ~  \K  \,  \circ d \phi (u_t, f_t) &=
 \int \phi_{u_t}^* \theta   \K\,  \circ d f_t + 
\int \phi_{f_t}^* \theta ~   \K  \, \circ d u_t.
\end{eqnarray}
 The second term on the right hand side vanishes since $\phi_{f_t}^* \theta ~   \K = 0$. The formula holds because  $ \phi_{u_t}^*  \K\, (z) = u_t z $ for all $z \in \R^n$.


 \eop

\noindent \textbf{Remark:} In the special case of $E=TM$, one can compare the classical integration of 1-forms with the covariant integration:  Let $Y_t$ be a semimartingale in $M$ and $v_t$ be a semimartingale in $E$.  If $v_t =u_t f_t$ such that $u_t$ is a horizontal lift of $Y_t$ and $f_t$ is the antidevelopment of $Y_t$, then for any 1-form $\theta$, the classical integration in $M$ and the covariant integration in $E$ coincides: 
\[
 \int \theta \, \circ dY_t = \int \theta \K \, \circ dv_t.
\]

\subsubsection*{Local coordinates:} Let $\{\delta_1, \ldots \delta_n \}$ be local sections in $E$ which is a basis in a coordinate neighbourhood $(U, \varphi = (x^1, \ldots , x^d))$, where $d$ is the dimension of $M$. For $1\leq \alpha, \beta \leq n $ and $1\leq i \leq d $, we write 
\[
 \nabla_{\frac{\partial}{\partial x^i}} \delta_{\alpha} = \Gamma_{i\alpha}^{\beta} \delta_{\beta},
\]
then
\[
 \K \left( \frac{\partial \delta_{\alpha}}{\partial x^i} \right) = \Gamma_{i\alpha}^{\beta} \delta_{\beta}.
\]
Let $\gamma_t$ be a differentiable curve in $M$ and  $u_t$ be a horizontal lift of $\gamma_t$ in $Gl(E)$, we write $u_t^{\beta}= u_t (e_ {\beta}) = u_t ^{\beta \alpha} \delta_{\alpha} (\gamma_t)$. Naturally 
\[
 \nabla _{\gamma'_t} u^{\beta}_t =0,
\]
and the parallel transport equation is given by
\[
 \frac{ u_t ^{\alpha \beta}}{dt} + \frac{\gamma^j}{dt} u_t ^{\alpha \gamma} \Gamma_{j \gamma}^{\beta} (\gamma_t) =0.
\]
For $\theta \in \Gamma(E^*)$, write $\theta = \theta^{\alpha} \delta^*_{\alpha}$, where $\theta ^\alpha = \theta (\delta_\alpha)$. We have, for each $1 \leq \alpha \leq n$
\[
  \left( \nabla_{\frac{\partial}{\partial x^j}} \theta \right) \delta^{\alpha} = \frac{\partial \theta^\alpha}{\partial x^j} - \theta (\nabla_{\frac{\partial}{\partial x^j}} \delta^{\alpha}) = \frac{\partial \theta^\alpha}{\partial x^j}- \Gamma_{j\alpha}^{\beta} \theta^\beta . 
\]
That is,
\[
 \nabla \theta = (\frac{\partial \theta^\alpha}{\partial x^j}- \Gamma_{j\alpha}^{\beta} \theta^\beta)\  dx^j \otimes \delta_\alpha^*.
\]

\subsubsection*{Cross quadratic variation in sections of  $TM^* \otimes E^*$:}

In order to find a covariant convertion formula for It\^o-Stratonovich integrals we introduce stochastic integration formulae for sections of  $TM^* \otimes E^*$, which is the space where the covariant derivative $\nabla \theta$ lives. Let  $v_t$ be a semimartingale  in $E$. Denoting  $x_t = \pi(v_t)$, we have the following identities: 
\begin{description}
\item[1)] For  $\alpha \in \Gamma (TM^* )$ and $\theta\in \Gamma (E^*)$,
\[
\int \alpha \otimes \theta \  (d  x_t, Dv_t) = \left<  \int \alpha \circ d \pi (v_t), \int \theta Dv_t\right>.
\]
\item[2)] For $b \in \Gamma (TM^* \otimes E^*)$ and $f\in C^{\infty} (M)$, 
\[
\int f b\  (d  x_t, Dv_t) = \int f(\pi (v_t)) \circ d \int b \  (d  x_t, Dv_t).
\]
\end{description}
This is well defined (similarly to  Emery \cite[p. 23]{Emery}). In particular, for $b=\nabla \theta$, in local coordinates:
\begin{eqnarray*}
\int \nabla \theta \  (d  x_t, Dv_t) &=& \int  (\frac{\partial \theta^\alpha}{\partial x^j}- \Gamma_{j\alpha}^{\beta} \theta^\beta)\ \circ d\int dx^j \otimes \delta_\alpha^*\  (d x_t, Dv_t) \\
 & = & \int  (\frac{\partial \theta^\alpha}{\partial x^j}- \Gamma_{j\alpha}^{\beta} \theta^\beta)(x_t) \ \circ d < x^j_t, \int u^{\gamma \alpha} d f^{\gamma} >_t \\
& = &   \int  (\frac{\partial \theta^\alpha}{\partial x^j}- \Gamma_{j\alpha}^{\beta} \theta^\beta) (x_t) u_t^{\gamma \alpha} \, \circ d <x^j, f^{\gamma} >_t.
\end{eqnarray*}

We have the following It\^o-Stratonovich covariant convertion formula:

\begin{proposition} Let $v_t$ be a semimartingale in $E$ and $\theta \in \Gamma (E^*)$. Then 

\begin{equation} \label{convertion}
 \int \theta  \ Dv_t = \int \theta D^I v_t + \frac{1}{2} \int \nabla \theta \  (d x_t, Dv_t).
\end{equation}
\end{proposition}

\proof In local coordinates we have that 
\[
\int  \theta \ Dv_t = \int \theta_{x_t} (u_t e_{\alpha}) \ \circ df_t^{\alpha} = \int \theta_{x_t} (u_t e_{\alpha}) \  df_t^{\alpha} + \frac{1}{2} <\theta (u e_{\alpha}), f^{\alpha}>.
\]
We have to show that 
\[
 \int \nabla \theta \  (d x_t, Dv_t) = <\theta (u e_{\alpha}), f^{\alpha}>.
\]
But
\begin{eqnarray*}
< \theta (u e_{\alpha}), f^{\alpha}> &=& < \theta_ x^{\beta} \delta^*_{\beta} (u e_{\alpha}), f^{\alpha}>\\
 & =& < \theta_ x^{\beta} (u e_{\alpha})^{\beta}, f^{\alpha}> \\
&=& \int   (ue_ {\alpha})^{\beta} \ d < \theta_ x^{\beta}, f^{\alpha}> + \int \theta ^{\beta}_ x \ d <(ue_ {\alpha})^{\beta}, f^{\alpha}> \\
&=& \int   u^{\alpha \beta} \ d < \theta_ x^{\beta}, f^{\alpha}> + \int \theta ^{\beta}_ x \ d <u^{\alpha \beta}, f^{\alpha}> \\
& =& \int   u^{\alpha \beta} \frac{\partial \theta^{\beta}}{\partial x^j}\ d < x^{j}, f^{\alpha}> + \int \theta ^{\beta}_ x \ d <u^{\alpha \beta}, f^{\alpha}> \\
& =& \int   u^{\alpha \beta} \frac{\partial \theta^{\beta}}{\partial x^j}\ d < x^{j}, f^{\alpha}> - \int \theta ^{\beta}_ x u^{\alpha \gamma} \Gamma_{j\gamma}^{\beta} (x) <x^j, f^{\alpha}> \\
& =& \int  \left( u^{\alpha \beta} \frac{\partial \theta^{\beta}}{\partial x^j} - \theta ^{\beta}_ x u^{\alpha \gamma} \Gamma_{j\gamma}^{\beta} (x) \right) <x^j, f^{\alpha}> \\
& =& \int  \left( u^{\alpha \gamma} \frac{\partial \theta^{\gamma}}{\partial x^j} - \theta ^{\beta}_ x u^{\alpha \gamma} \Gamma_{j\gamma}^{\beta} (x) \right) <x^j, f^{\alpha}> \\
& =& \int  \left(  \frac{\partial \theta^{\gamma}}{\partial x^j} - \theta ^{\beta}_ x \Gamma_{j\gamma}^{\beta} (x) \right) u^{\alpha \gamma} <x^j, f^{\alpha}> \\
&=& \int \nabla \theta \ (dx_t, Dv_t).
\end{eqnarray*}
  
\eop

\subsubsection*{It\^o representation:}
The vertical lift of an element $w \in E$ to the tangent space $T_eE $, with $e$ and $w$ in the same fiber is given by  
\begin{equation} \label{verticalcomponent}
w^v  =\frac{d}{dt} [e+tw]_{t=0} \in T_eE.
\end{equation}
Let $r, s$ be sections of $E$ and $X, Y$ be vector fields of $M$. We shall consider a connection  $\nabla^h$  in $E$, a prolongation of $\nabla$, which satisfies the following:
\[
\begin{array}{lcl}
\nabla^h_{r^v}s^v=0, & & \nabla^h_{X^h}s^v=(\nabla_X s)^v, \\

 & & \\
 \nabla^h_{r^v}Y^h=0, &   & \nabla_{X^h} Y^h \mbox{ is horizontal.}
\end{array}
\]

\bigskip

\noindent \textbf{Remark:} An example of this connections is the horizontal connection defined by Arnaudon and Thalmaier \cite{Arnaudon-Thalmaier}, where, considering a connection $\tilde{\nabla}$ in $M$, the extra condition $\nabla^h_{X^h}Y^h = (\tilde{\nabla}_X Y)^h $ characterizes this connection.

\bigskip

Next proposition shows a  geometrical characterization of the covariant It\^o integral.

\begin{proposition} Let $v_t$ be a semimartingale in $E$ and $\theta \in \Gamma (E^*)$. Then 
\begin{equation} \nonumber  
 \int \theta\ D^I v_t = \int \theta\,  \K\ d^{\nabla^h} v_t. 
\end{equation}
 
\end{proposition}

\proof We have to calculate each component of $\nabla^h \theta \,  \K$. Using that for $A, B$ vector field in $E$ we have that 
\[
 \nabla_A^h \theta \,  \K (B) = A (\theta \,  \K (B)) - \theta \,  \K (\nabla^h_A B),
\]
 we obtain the components
\[
\begin{array}{lcc}
\nabla^h_{r^v} \theta \,  \K (s^v)  =  0,  & ~ ~ &   
\nabla^h_{r^v} \theta \,  \K (Y^h)  =  0, \\
 & \\
\nabla^h_{X^h} \theta \,  \K (s^v)  =  \nabla_X  \theta (s) \circ \pi, & ~ ~ &
 \nabla^h_{X^h} \theta \,  \K (Y^h)  =  0.
\end{array}
\]
Hence, using It\^o- Stratonovich convertion formula for classical 1-form integration, see e.g. Catuogno and Stelmastchuk \cite{Catuogno-Stelmastchuk}:
\begin{eqnarray*}
 \int \theta\ Dv_t &=& \int \theta \,  \K \circ dv_t \\
                  & = & \int \theta \,  \K\ d^{\nabla^h} v_t + \frac{1}{2} \int \nabla^h \theta \,  \K (dv_t, dv_t).
\end{eqnarray*}
For the correction term, we have that:
\[
 \nabla^h \theta \,  \K = \nabla \theta \,  (\pi_* \times \K),
\]
in the sense that $\nabla^h_{\pi_*A} \theta \K (B) = \nabla \theta \,  (\pi_* \times \K) (A,B)$. But
\[
  \int \nabla^h \theta \,  \K (dv_t, dv_t) = \int \nabla \theta (dx_t, Dv_t).
\]
Combining with equation (\ref{convertion}), we have that
\[
  \int \theta\ D^I v_t = \int \theta \,  \K\ d^{\nabla^h} v_t.
\]

\eop

\subsubsection*{Vector bundle mappings}

Consider two vector bundles $\pi: E \rightarrow M$,  $\pi': E' \rightarrow M'$ and a differentiable fibre preserving mapping $F: E \rightarrow E'$ over a differentiable map $\widetilde{F}: M \rightarrow M'$, i.e.  $ \pi '\circ F = \widetilde{F} \circ \pi $.

Let $\K$ and $\K '$ be connectors in $E$ and in $E'$ respectively.
We define the vertical derivative (or derivative in the fibre)  of $F$ in the direction of $w$ by:
\[
 D^v F (e) (w) = \K'F_* (w^v),
\]
where the vertical component $w^v$ is given by Equation (\ref{verticalcomponent}). For $Z\in T_{\pi(e)}M$, the horizontal (or parallel) derivative is:
\[
 D^h F (e) (Z)= \K' F_* (Z^h).
\]
For a vector field  $X$ in $E$, we have that
\[
 X=(\pi_* X)^h+ \K(X),
\]
hence 
\begin{equation} \label{decomposicao}
\K'F_* (X) = D^vF( \K (X) ) + D^hF ( \pi_* (X)).
\end{equation}

The It\^o formula for the Stratonovich covariant integration includes an usual 1-form integration, compare with Norris \cite[Eq. (20)]{Norris}:

\begin{proposition} \label{Pullback Norris-Stratonovich}
Given a fibre preserving map $F$ as above,
 \begin{equation}   \label{formulaN-S}
  \int \theta DF (v_t) = \int (D^v F)^* \theta Dv_t + \int (D^hF)^* \theta \ \circ d(\pi v_t). 
 \end{equation}
\end{proposition}
 
\proof We just have to use the decomposition of Equation (\ref{decomposicao}). 
\begin{eqnarray*}
 \int \theta DF (v_t) &=& \int \theta\K'F_* \ \circ d v_t \\
 &=& \int \left( \theta D^vF \K  + \theta D^hF  \pi_*  \right) \circ dv_t \\
 &=& \int (D^v F)^* \theta Dv_t + \int (D^hF)^* \theta \ \circ d(\pi v_t).
\end{eqnarray*}

\eop

\begin{proposition} For a section $b'$ in $(TM')^* \otimes (E')^*$ and a fibre preserving map $F: E \rightarrow E '$ over $\widetilde{F}: M \rightarrow M'$
 we have that
\[
 \int b'(d \pi 'F (v_t), DF(v_t))= \int ( \widetilde{F}_* \otimes D^v F )^* b' (d \pi v_t, D v _t) +  \int ( \widetilde{F}_* \otimes D^h F )^* b' (d \pi v_t, d \pi v _t). 
\]
\end{proposition}

\proof We have 

\begin{eqnarray*}
  \int b'(d \pi 'F (v_t), DF(v_t)) & =& \int b' \,  (\pi'_*\otimes \K') (d F (v_t), d F (v_t)) \\
 &=& \int b' \,  (\pi'_* \otimes \K')\,  (F_* \otimes F_*) (d v_t, d  v_t). \\
\end{eqnarray*}
Using that 
\[ (\pi'_* \otimes \K')\,  (F_* \otimes F_*) = \widetilde{F}_* \,  \pi_* \otimes (
D^vF \K  + D^hF  \pi_* ) 
\]
 yields 
\[
 \int b'(d \pi 'F (v_t), DF(v_t))= \int ( \widetilde{F}_* \otimes D^h F )^* b' (d \pi v_t, D v _t) +  \int ( \widetilde{F}_* \otimes D^h F )^* b' (d \pi v_t, d \pi v _t).
\]

\eop

It\^o version of Formula (\ref{formulaN-S}) is given by:

\begin{proposition} Given a fibre preserving map $F$ as above,
 \begin{eqnarray*}
  \int \theta D^I F(v_t)& =& \int (D^v F)^* \theta D^I v_t + \int (D^hF)^* \theta \circ d\pi v_t + \\
 && \\
 & & \frac{1}{2} \int \left( \nabla (D^v F ^* \theta)- (\widetilde{F}_* \otimes
D^VF)^* \nabla' \theta \right) (\ d \pi v_t, D v_t )+\\
 && \\
 & & \frac{1}{2} \int \left(  \widetilde{F}_* \otimes D^h F \right)^* \nabla'\theta (d\pi v_t, d\pi v_t). 
 \end{eqnarray*}
\end{proposition}
 
\proof By Proposition 2.2 we have that 
\[
  \int \theta D^I F(v_t) = \int \theta D F (v_t) - \frac{1}{2} \int \nabla '
\theta (d\pi'F(v_t), DF(v_t)) \\
\] 
and 
\[
 \int (D^v F)^* \theta Dv_t = \int (D^v F)^* \theta D^I v_t + \frac{1}{2} \nabla (D^v F)^* (d \pi v_t, Dv_t).
\]
But, Proposition 2.4 says that:
\[
\int \theta DF (v_t) = \int (D^v F)^* \theta Dv_t + \int (D^hF)^* \theta \ \circ d(\pi v_t).  
\]
Finally, by Proposition 2.5, we have that
\begin{eqnarray*}
\int \nabla'\theta (d \pi 'F (v_t), DF(v_t))&=& \int ( \widetilde{F}_* \otimes
D^v F )^* \nabla'\theta (d \pi v_t, D v _t) +  \\
&& \\
&&  \int ( \widetilde{F}_* \otimes
D^h F )^* \nabla '\theta (d \pi v_t, d \pi v _t),
\end{eqnarray*}
which implies the formula.

\eop

\section{Applications}

\subsubsection*{Commutation Formulae}

Given a differentiable map $(a,b)\in \R^2 \mapsto E $, let $s_E: TTE \rightarrow TTE$ be the symmetry  map given by $s_E (\partial_a \partial_b s(a,b)) = \partial_b \partial_a s(a,b)$. Let $C = \mathcal{K} \,  \mathcal{K}_* - \mathcal{K} \,  \mathcal{K}_* \,  s_E: TTE \rightarrow E$ be the curvature of $\mathcal{K}$. If $u, v \in TM$ and $s \in \Gamma (E)$ then the relation between the curvature of $\mathcal{K}$ with the curvature of the connection $\nabla$ is given by $R^E (u, v)s = C(u v s)$,  see Paterson \cite{Paterson}.

Let  $I \subset\R $ be an open interval and  consider $a\in I \mapsto J(a)$ a differentiable 1-parameter family of semimartingales in $E$. Then 
\begin{eqnarray*}
 \int \theta ~  D \nabla_a J 
& = & \int \theta \K \   ~\circ d  (\nabla_a J) \\
& =& \int \theta ~  \K \   ~\circ d  \K \partial_a J \\
 & =&  \int \theta \K \ \,   \mathcal{K}_{\nabla *}  ~\circ d   \partial_a J \\
 & =& \int \theta ~  \K \ \,   \mathcal{K}_{\nabla *}  ~\circ d   \partial_a J - \int \theta \K \,  \mathcal{K}_{\nabla *} \,  s_E  ~ \circ d  \partial_a  J\\
&& + \int \theta ~  \K \,  \mathcal{K}_{\nabla *} \,  s_E ~  \circ d  \partial_a J\\
& =& \int C   ~\circ d   \partial_a J  + \int \theta \nabla_a DJ.
\end{eqnarray*}

\noindent Compare with Arnaudon and Thalmaier \cite[Equation 4.13]{Arnaudon-Thalmaier}.
An It\^o version, as in \cite{Arnaudon-Thalmaier} can be obtained by convertion formulae.

\subsubsection*{Harmonic sections}

 Let $M$ be a Riemannian manifold and $\pi:V\rightarrow M$ be a Riemannian vector bundle with a connection $\nabla$ which is compatible with its metric. We denote by $E^p$ the vector bundle $\bigwedge^p T^*M\otimes V$ over $M$.  In this context, we shall consider three differential geometric operators. The exterior differential operator  $d:\Gamma(E^p)\rightarrow\Gamma(E^{p+1})$ is defined by
\[
 d\sigma(X_1,\ldots, X_{p+1}):=(-1)^k(\nabla_{X_k}\sigma)(X_0,\ldots,\hat{X}_k,\ldots,X_p).
\]
The co-differential operator  $\delta:\Gamma(E^p)\rightarrow\Gamma(E^{p-1})$ is defined by
\[
 \delta\sigma(X_1,\ldots, X_{p-1}):=-(\nabla_{e_k}\sigma)(e_k,X_1,\ldots,X_{p-1}),
\]
where $\{e_i\}$ is a local orthonormal frame field. And the Hodge-Laplace operator $\Delta : \Gamma(E^p) \rightarrow \Gamma (E^p)$ is given by 
\[
 \Delta =(d\delta+\delta d).
\]
One of the cornerstones of modern geometric analysis is the Weitzenb\"ock formula which states that 
\[
 \Delta \sigma=-\nabla^2\sigma+\Phi(\sigma),
\]
for a  $\Phi\in\textrm{End}(E^p)$, see e.g.  Eells and Lemaire \cite[p.11]{Eells-Lemaire} or Xin \cite[p.21]{xin}.


Let  $B_t$ be a Brownian motion in $M$ and $e_t\in \textrm{End}(E^p)$ be the solution of 
\[
 D^Ie_t=e_t \circ \Phi(B_t)~dt.
\]


\begin{theorem} A section $\sigma \in \Gamma(E^p)$ is harmonic (i.e. $\Delta \sigma=0$) if and only if for any $\theta\in\Gamma(E^{p*})$
\[
 \int\theta~ D^I\sigma_t
\]
is a local martingale, where $\sigma_t= e_t\sigma(B_t)$.

\end{theorem}
\proof The result now is consequence of Weitzenb\"ock formula and the following

\begin{lemma}
 Consider $\sigma\in\Gamma(E^p)$, $\theta\in\Gamma(E^{p*})$ and a semimartingale $x_t$ in $M$.  Given  $V\in\textrm{End}(E^p)$, let $e_t \in\textrm{End}(E^p)$ be the solution of 
\[
 D^Ie_t=e_t \circ V(x_t)~g(d x_t,d x_t).
\]
Write $\sigma_t=e_t \sigma(x_t)$. Then
\[
 \int \theta D^I\sigma_t=\int (\theta\circ\nabla\sigma)~d^{\nabla^M}x_t+\int (\theta\circ e_t)\left(\frac{1}{2}\nabla^2 + V(\sigma (x_t))  g \right)\sigma~(dx_t,dx_t).
\]
\end{lemma}
\begin{proof} By covariant It\^o-Stratonovich convertion formula, Equation (\ref{convertion}), we have that
 \begin{eqnarray}
  \int\theta D^I\sigma_t&=&\int\theta~D^Ie_t(\sigma(x_t))+\int(\theta\circ e_t)~D^I(\sigma(x_t))\nonumber\\
&=&\int (\theta\circ e_t)~V(\sigma(x_t))~g(d x_t,d x_t)+\int (\theta\circ e_t)~D^S(\sigma(x_t))\nonumber\\
&&+\frac{1}{2}\int \nabla(\theta\circ e_t)~(d x_t,D\sigma(x_t)).\label{A3}
 \end{eqnarray}
Now, by usual It\^o-Stratonovich convertion formula: 
\begin{eqnarray}
 \int (\theta\circ e_t)~D^S(\sigma(x_t))&=&\int (\theta\circ e_t)\mathcal{K}_\nabla \, \sigma_*~ d x_t\nonumber\\
&=&\int (\theta\circ e_t)\nabla\sigma~ d^{\nabla^M} x_t\nonumber\\
&&-\frac{1}{2}\int\nabla^M(\theta\circ e_t\circ\nabla\sigma)(d x_t, d x_t).\label{A1}
\end{eqnarray}
We have that 
\begin{eqnarray}
 \int \nabla(\theta\circ e_t)~(d x_t,D \sigma (x_t) )=\int \nabla(\theta\circ e_t)\circ(I\otimes\nabla\sigma)~(d x_t,d x_t)\label{A2}
\end{eqnarray}
substituing (\ref{A1}) and (\ref{A2}) in (\ref{A3}) one finds: 
\begin{eqnarray*}
  \int\theta D^I\sigma_t&=& \int (\theta\circ e_t)~V(\sigma(x_t))~g(d x_t,d x_t)+\int (\theta\circ e_t)\nabla\sigma~d^{\nabla^M} x_t\nonumber\\
&&-\frac{1}{2}\int\nabla^M(\theta\circ e_t\circ\nabla\sigma)(d x_t,d x_t) \nonumber\\
&&+\frac{1}{2}\int \nabla(\theta\circ e_t)\circ(I\otimes\nabla\sigma)~(d x_t,d x_t).
\end{eqnarray*}
The result follows using that for all $\theta\in\Gamma(E^*)$,
\[
 \nabla\theta\circ(I\otimes\nabla\sigma)-\nabla^M(\theta\circ\nabla\sigma) =\theta(\nabla^2\sigma).
\]

\eop

\end{proof}

\end{document}